\documentclass[a4paper,11pt]{article}
\usepackage[plainpages=false]{hyperref}
\usepackage{amsfonts,latexsym,rawfonts,amsmath,amssymb,amsthm,mathrsfs}
\usepackage{amsmath,amssymb,amsfonts,latexsym,lscape,rawfonts}

\usepackage[all]{xy}
\usepackage{eufrak}
\usepackage{makeidx}         
\usepackage{graphicx,psfrag}

\usepackage{array,tabularx}

\usepackage{setspace}

\newtheorem{thm}{Theorem}[section]

\newtheorem{lem}[thm]{Lemma}
\newtheorem{clm}[thm]{Claim}
\newtheorem{prop}[thm]{Proposition}

\theoremstyle{remark}
\newtheorem{rmk}[thm]{Remark}

\theoremstyle{definition}
\newtheorem{Def}[thm]{Definition}                                        %

\def \C {\mathbb C}

\def \R {\mathbb R}

\title{ $C^{2,\alpha}$-estimate for     Monge-Ampere equations with H\"older-continuous right hand side}
\author{Xiuxiong Chen,\ Yuanqi Wang}
\date{}
\begin{document}

\maketitle{}
\begin{abstract}
We present a somewhat new proof  to the $C^{2,\alpha}$-aprori estimate  for the uniform  elliptic Monge-Ampere equations, in both the real and complex settings. Our estimates do not need to differentiate the equation, and    only depends on the $C^{\alpha^{\prime}}-$norm of the right hand side of the equation, $0<\alpha<\alpha^{\prime}$.  
\end{abstract}

\section{Introduction}
Given an uniformly elliptic Monge-Ampere equation, historically,  there are various methods to obtain higher order estimates.
One of the pioneering work is the celebrated third derivatives by E. Calabi \cite{Calabi58}, where he requires the solution is of class $C^5.\;$
In 1980s,  from a complete different point of view,   Evans-Krylov-Safonov \cite{Evans1},\cite{Evans2},\cite{Krylov}   gave the famous Schauder estimate.  While their method requires
less regularity,  their estimates  still  rely on differentiating both sides of the equation. Thus,  the estimates they give depend on $W^{1,p}-$norms (for $p$ big) or higher derivative norms of the right-hand side of the  equation. 

 Later, Safonov \cite{Safonov} and  Caffarelli \cite{Caff}  discovered the following  celebrated    
  $C^{2,\alpha}-$estimate without  differentiating the Monge-Ampere equations. 
\begin{thm}\label{thm of Caffarelli}   Suppose $B$ is a unit ball in $\mathbb{R}^n$, and
$u\in C^{2,\alpha^{\prime}}(\bar{B})$ is a convex   function.  Suppose
\begin{equation}\label{equ real MA}
\det u_{ij} =  e^f > 0,\ \textrm{where}\ f \in C^{\alpha^{\prime}} (\bar{B}).
\end{equation}
 Then, for any $\alpha \in (0,\alpha^{\prime}),$  we have
\[
[\nabla^{2}u]_{\alpha,{B\over 4}} \leq C,
\] 
where $C$  depends on the $C^{0}(\bar{B})-$norm of    $\nabla^{2}u$,  the $C^{\alpha^{\prime}}(\bar{B})-$norm of $f$, the dimension $n$, the $\alpha^{\prime}$, and $\alpha$. 
\end{thm}
\begin{rmk} For $\alpha\in [0,1]$,  $|\cdot|_{\alpha,\Omega}$ ( $[\ \cdot\ ]_{\alpha,\Omega} $) means the $C^{\alpha}(\bar{\Omega})-$norm (seminorm) in the domain $\Omega$, $|\cdot|_{k,\alpha,\Omega}$ ( $[\ \cdot\ ]_{k,\alpha,\Omega} $) means the $C^{k,\alpha}(\bar{\Omega})-$norm (seminorm). Most of the
notations here follow the conventions in \cite{GT}.
\end{rmk}
\begin{rmk}  One of the  main features of Safonov and  Caffarelli's results  is that the $C^{2,\alpha}-$norm of the solution only depends mainly on the $C^{\alpha^{\prime}}$ of the right hands side, provided the $C^{2}-$estimate is already obtained (or equivalently, the equation is uniformly elliptic). 
\end{rmk}
\begin{rmk}The above mentioned Schauder estimates are never a complete list of existing beautiful estimates of this kind. Historically,  on the Schauder estimates on nonlinear uniformly-elliptic equations, we also have the work  of C, Burch  \cite{Burch},  J, Kovats \cite{Kovats},  Q,B, Huang \cite{QBHuang}. More recently,  X,J Wang \cite{XJWang} gives a nice 
Schauder estimate for both linear and nonlinear equations. For more work on Schauder estimates, we refer to the readers to \cite{XJWang} and the references therein. 
\end{rmk}

 The corresponding theory of Safonov and Caffarelli's results  in complex settings  also has important progress in recent years.  Assuming full second  derivative bound, Dinew-Zhang-Zhang \cite{DZZ} showed a $C^{2,\alpha}-$estimate for complex Monge-Ampere equations, which only depends on the H\"older continuity of the right hand side. Very recently, a theorem to the same strength as Safonov and Caffarelli's 
was proved by Yu Wang \cite{YuWang}, which   relies  on a clever trick to  convert the complex Monge-Ampere equation to a real equation, then apply Caffarelli's more general estimates in \cite{Caff}.

 In this note, following  \cite{CDS2}, \cite{Anderson}, we present   another proof  of Theorem \ref{thm of Caffarelli} and its complex analogue Theorem \ref{thm of Yu Wang}.  
 While we believe our presentation/proof contains some new element (i.e., input from Riemannian geometry), the idea of rescaling and Liouville type
 theorem in general goes back to long time ago, for example, see Leon Simon's beautiful proof of Schauder estimates for linear operators \cite{LeonSimon}. For this reason, we hope that our proof  to this classical theorem on
 full nonlinear PDEs, is still somewhat valuable.
 
\begin{thm}\label{thm of Yu Wang}   Suppose $B$ is a unit ball in $\mathbb{C}^n$, and
$\phi\in C^{2,\alpha^{\prime}}(\bar{B})$ is a pluri-subharmonic  function.  Suppose
\begin{equation}\label{equ MA}
\det \phi_{i\bar j} =  e^f > 0,\ \textrm{where}\ f \in C^{\alpha^{\prime}} (\bar{B}).
\end{equation}
 Then, for any $\alpha \in (0,\alpha^{\prime}),$  we have
\[
[\sqrt{-1}\partial \bar{\partial}\phi]_{\alpha,{B\over 4}} \leq C,
\] 
where $C$  depends on the $C^{0}(\bar{B})-$norm of  $\sqrt{-1}\partial \bar{\partial}\phi$,  the $C^{\alpha^{\prime}}(\bar{B})-$norm of $f$, the dimension $n$, the $\alpha^{\prime}$, and $\alpha$. 
\end{thm}
\begin{rmk} Theorem \ref{thm of Yu Wang} is  similar to but  slightly different from Yu Wang's theorem in \cite{YuWang}. First, Theorem \ref{thm of Yu Wang} is an aprori estimate (it assumes the solution is in $C^{2,\alpha^{\prime}}$), while the theorem of Yu Wang is a stronger regularity theorem. Second, the norm bound in Theorem \ref{thm of Yu Wang} does not depend on $|\phi|_{L^{\infty}}$ i.e the lower order bound on the potential. This is essentially because equation (\ref{equ MA}) is a geometric equation for the K\"ahler-metric $\sqrt{-1}\partial\bar{\partial} \phi$.  Given the importance of the Calabi's, Evan-Krylov-Safonov's, Safonov's, and  Caffarelli's Schauder estimates, we hope our new proof is worthwhile to present  separatedly here.
\end{rmk}

We hope our new proof   has further applications in fully nonlinear equations on singular spaces. Actually, our new proof (in section 2) is developed in the proof of 
Theorem 1.7 in \cite{ChenWanglongtime} on K\"ahler-Ricci flows with conic singularities.   The purpose of this note is to give a more simplified 
and direct proof than the one in \cite{ChenWanglongtime} (from page 13 to page 19), and to 
show our method also works for real Monge-Ampere equations. Namely, the following theorem in essentially proved in \cite{ChenWanglongtime} (from page 13 to page 19).
   \begin{thm}\label{Thm elliptic C2alphabeta estimate for conic MA equations}  Suppose $\beta\in (0,1)$ and $0<\alpha^{\prime}<\min \{\frac{1}{\beta}-1,1\}$. Suppose $\bar{B}$ is the unit ball centered at the origin with respect to the model cone metric 
    $$\omega_{\beta}=\sqrt{-1}\frac{\beta^2}{|z|^{2-2\beta}} dz\wedge d\bar{z}+\sqrt{-1}\Sigma_{k=2}^{n}dv_{k}\wedge d\bar{v}_{k},$$ which is defined over $\C\times \C^{n-1}$ with cone singularity of angle $2\beta\pi$ along the divisor $\{0\}\times \C^{n-1}$.
 Suppose $\phi$ is a pluri-subharmonic function in $C^{2,\alpha^{\prime},\beta}(\bar{B})$ such that 
$$\frac{1}{K}\omega_{\beta}\leq \sqrt{-1}\partial\bar{\partial}\phi\leq K\omega_{\beta}\ \textrm{ over} \ \bar{B}\setminus D\ \textrm{for some}\  K\geq 1.$$

Denote $F_{\phi}$ as $\log(|z|^{2-2\beta}\det \phi_{i\bar j})$, which means  $$\det \phi_{i\bar j} =  \frac{e^{F_{\phi}}}{|z|^{2-2\beta}}\ \textrm{over}\  \bar{B}\setminus D.$$
Then for any $\alpha\in (0,\alpha^{\prime})$, there exists a constant $C$ depending on $|F_{\phi}|_{\alpha^{\prime},\beta,B}$, $K$, $n$, $\alpha$, $\alpha^{\prime}$, and $\beta$, such that \[[\sqrt{-1}\partial \bar{\partial}\phi]_{\alpha,\beta,\frac{B}{4}}\leq C.\]
\end{thm}
 The $C^{2,\alpha,\beta}$ and $C^{\alpha,\beta}$ function spaces are defined by Donaldson in \cite{Don}.  For further references on definition of these function spaces,  see \cite{ChenWang}, \cite{ChenWanglongtime}, \cite{WYQWF}.

 Since this is a short note in the smooth case, we  will not go into the ever-growing list of works in conical settings, instead we  refer the readers to the following  list of authors and their work    related to  the $C^{2,\alpha}-$estimate in conical K\"ahler geometry: Donaldson  \cite{Don}, Brendle \cite{Brendle},  Guenancia-Paun \cite{GP}, Chen-Donaldson-Sun\cite{CDS2}, Jeffres-Mazzeo-Rubinstein \cite{JMR}, Calamai-Zheng \cite{CalamaiZheng}...

\textbf{Acknowledgement:} Both authors are grateful to Yu  Yuan for useful suggestions on  earlier versions of this paper. The second author would like to thank Weiyong He for related discussions. 

\section{ A new proof of the aprori version of Yu-Wang's Cafferelli type estimate for complex Monge-Ampere equations. \label{A new proof of the aprori version of Yu-Wang's Cafferelli type estimate}}
Our proof is based on Anderson's rescaling idea in \cite{Anderson}  and Chen-Donaldson-Sun's trick in \cite{CDS2}.

  In K\"ahler geometry, a  K\"ahler metric usually means a closed positive $(1,1)-$form. Give a $\phi$ as in Theorem \ref{thm of Yu Wang}, 
  $\sqrt{-1}\partial \bar{\partial} \phi$ is then a K\"ahler-metric. In general, given a K\"ahler-metric $\omega$ in an open set $\Omega$,  a pluri-subharmonic function $\phi$  is said to be a potential of  $\omega$ in  $\Omega$ if 
\[\omega=\sqrt{-1}\partial \bar{\partial}\phi,\ \textrm{or equivalently}\ \omega_{k\bar{l}}=\frac{\partial^2 \phi}{\partial z_{k}\partial \bar{z}_{l}}\ \textrm{over}\ \Omega,\]
where $\omega_{k\bar{l}}$ is defined as $\omega=\sqrt{-1}\omega_{k\bar{l}}dz_{k}\wedge d\bar{z}_{l}$.
Under the coordinates $z_{1},...,z_{n}$, $\omega_{k\bar{l}}$ ($\frac{\partial^2 \phi}{\partial z_{k}\partial \bar{z}_{l}}$) is a Hermitian-matrix-valued function.   In the rest of this section, we work with the metrics $\omega$ most of time rather than the potentials $\phi$. This is because our proof is essentially Riemannian geometry.

Our proof depends on the  following 3 building blocks.  

(1): The solvability of $\sqrt{-1}\partial \bar{\partial}-$equation with tame estimates. 
 \begin{lem}\label{lem dbar solvability smooth case} Suppose $\frac
{1}{\Lambda}<r<\Lambda$ for some $\Lambda>0$, then there exists a constant $C_{\Lambda}$ depending on $\Lambda$, $n$, and $\alpha$ with the following 
properties. 

 Suppose  $\eta\in C^{\alpha}(B_{r})$ is a closed real $(1,1)$-form. Then there exists a  real valued solution $\varphi\in C^{2,\alpha}(B_{\frac{r}{2}})$ to
 
\begin{equation}\label{equ ddbar equation smooth}\sqrt{-1}\partial \bar{\partial}\varphi=\eta\ \textrm{over}\ B_{\frac{r}{2}}
\end{equation} 
such that $|\varphi|_{0,B_{\frac{r}{2}}}\leq Cr^2|\eta|_{0, B_{r}},$
where $C$ is constant depending on $n$. Consequently, 
\[|\varphi|_{2,\alpha,B_{\frac{r}{4}}}\leq C_{\Lambda}|\eta|_{\alpha,B_{r}}.\]
\end{lem}
\begin{rmk} Lemma \ref{lem dbar solvability smooth case} is an simpler version of Lemma 7.1 in \cite{ChenWangRegularity}. 
\end{rmk} 

(2): The  Liouville theorem in the complex case. 
\begin{thm}\label{thm Liouville}(Riebesehl; Schulz)(\cite{RS}) Suppose $\omega$ is a K\"ahler-metric  defined over $\C^{n}$ which admits a $C^{2,\alpha}-$potential over any finite ball. Suppose there is a constant  $K$ such that  
\begin{equation}
det\omega_{k\bar{l}}=1,\ \frac{1}{K}I\leq \omega_{k\bar{l}}\leq KI\ \textrm{over}\ \C^{n}.
\end{equation}
Then, for any $1\leq k,l\leq n$, $\omega_{k\bar{l}}$ is a constant. 
\end{thm}
\begin{rmk}By the proof of Lemma \ref{lem dbar solvability smooth case} (see Lemma 7.1 in \cite{ChenWangRegularity}),  the $\omega$ in Theorem \ref{thm Liouville} actually admits a global potential $\phi$ over $\C^{n}$, thus Theorem \ref{thm Liouville} is actually the same as Theorem 2 in \cite{RS}. However, we would like to emphasize that to prove  Theorem \ref{thm Liouville},  we don't need to assume the metric admits a global potential. Therefore, Theorem \ref{thm Liouville} can be  carried   over exactly and directly to the proof of section \ref{A new proof of the aprori version of  Cafferelli's  estimate} in the real case, without involving more issues. 
\end{rmk}

(3): The Chen-Donaldson-Sun's trick in \cite{CDS2}. This following version  is proved simply by using Lemma \ref{lem dbar solvability smooth case} to inequality (39) in \cite{CDS2} (with lower order term changed from $[\phi]_{\alpha}$ to  $|\phi|_{0}$).
\begin{prop}\label{prop CDS trick}For any constant coefficient K\"ahler metric $\omega_{c}$, there exist a small enough positive number $\delta$ and a big enough constant $C_{\omega_{c}}$, both depending on the positive lower and upper bounds on the eigenvalues of $\omega_{c}$, the dimension $n$, and $\alpha^{\prime}$, with the following properties. Suppose $\omega$ is a  K\"ahler-metric   over
$B_{0}(1)$ which admits a potential in $C^{2,\alpha^{\prime}}[B_{0}(1)]$. Suppose  
\begin{equation}
det\omega_{k\bar{l}}=e^{f},\  \frac{\omega_{c}}{1+\delta}\leq \omega\leq (1+\delta)\omega_{c}\ \textrm{over}\ B_{0}(1),
\end{equation}
then the following estimate holds in $B(\frac{1}{4})$. 
\[[\omega]_{\alpha^{\prime}, B(\frac{1}{4})}\leq 
C_{\omega_{c}}[|e^f|_{\alpha^{\prime},B(1)}+|\omega|_{0,B(1)}].\]
 
\end{prop}

 Now we are ready to prove of Theorem \ref{thm of Yu Wang}.
\begin{proof}{of Theorem \ref{thm of Yu Wang}:} In this proof, while different "C" can be different constants,  the dependence of each $"C"$ is the same as the last sentence of  Theorem \ref{thm of Yu Wang}. We add more subletter to $C$ if it depends on more things.

 Notice that by  the Monge-Ampere equation (\ref{equ MA}), the $C^{0}-$norms of  $\sqrt{-1}\partial \bar{\partial}\phi$ and $f$ in Theorem \ref{thm of Yu Wang} determines a $K\geq 1$ such that 
\begin{equation}\label{equ quasi isometric constant K in new proof }
\frac{I}{K}\leq \sqrt{-1}\partial \bar{\partial}\phi\leq KI. 
\end{equation}

 Our proof can be divided into $3$ steps. 

Step 1: The notion of H\"older-radius and contradiction hypothesis. 

Denote $\omega_{Euc}$ as the Euclidean metric in the underline coordinates, and  $d_{q}=dist_{\omega_{Euc}}(q,\partial B(1))$.
\begin{Def}
H\"older Radius: given a closed $(1,1)-$form $\omega\in C^{\alpha^{\prime}}[\bar{B}(1)]$, for all $q\in B(1)$, we define $h_{\omega,q}$ as the supremum of the  radiuses $h\in (0,d_{q})$ with the following properties. \end{Def}
\begin{equation}\label{equ definition Holder radius}
 [\omega]_{\alpha,B_{q}(h)}=\Sigma_{k,l}[\omega_{k\bar{l}}]_{\alpha,B_{q}(h)}\leq \delta_{0}h^{-\alpha},
\end{equation}
where $\delta_{0}>0$ is  small enough with respect to the data in the last sentence of Theorem \ref{thm of Yu Wang}. Notice definition (\ref{equ definition Holder radius}) depends on the coordinates, thus when we rescale the coordinates, (\ref{equ MA equation holds in rescaled coordinates}) and (\ref{equ holder radius at the origin is 1 after rescaling}) hold. Since $\omega=\sqrt{-1}\partial \bar{\partial}\phi$  is assumed to be $C^{\alpha^{\prime}}$, and we are considering open balls,  then actually the supremum of radiuses can be attained. However, we don't need the H\"older radius to be attainable  in our proof.

 By definition, we obtain the following extremely simple but extremely important property of the H\"older radius. 
\begin{clm}\label{clm good property of the radius complex case}   For any $0<r<h_{q}$, we have 
$[\omega]_{\alpha,B_{q}(r)}\leq \delta_{0}r^{-\alpha}.$
\end{clm}

To prove Theorem \ref{thm of Yu Wang}, it suffices to show 
\begin{equation}\label{equ what we wanna show in contra positive proof}
\frac{h_{\omega,q}}{d_{q}}\geq c_{1}>0,\ \textrm{for some}\ c_{1}\ \textrm{depending only on}\ K\ \textrm{and}\ |f|_{\alpha^{\prime},B(1)}. 
\end{equation}

We  prove by contradiction. Were (\ref{equ what we wanna show in contra positive proof}) not true,  there exists a sequence of K\"ahler metrics
$\omega_{i}=\sqrt{-1}\partial \bar{\partial}\phi_{i}$, functions $F_{i}$, and points $q_{i}$ such that 
\begin{eqnarray}\label{eqn contradiction hypothesis}
&\bullet& det \omega_{i,k\bar{l}}=e^{F_{i}}\ \textrm{over}\ B(1);\nonumber
\\& \bullet& \frac{\omega_{Euc}}{K}\leq \omega_{i}\leq K\omega_{Euc},\ K\geq 1,\
|F_{i}|_{\alpha^{\prime}, B(1)}\leq c;\nonumber
\\&\bullet & \textrm{for any fixed}\ i, \frac{h_{\omega_{i},q_{i}}}{d_{q_{i}}}=\epsilon_{i}>0\ (\textrm{since we are doing aprori estimate});\nonumber
\\&\bullet& \frac{h_{{\omega}_{i}, q_{i}}}{d_{q_{i}}}=\epsilon_{i}\rightarrow 0, \     \frac{h_{\omega_{i}, q_{i}}}{d_{q_{i}}}\leq 2 \min \frac{h_{\omega_{i}, q}}{d_{q}},\ \textrm{for any}\ q\in B(1).
\end{eqnarray}

We shall derive a contradiction.

Step 2: Rescaling, norm bounds, and bootstrapping.

 We consider the rescaling
\begin{itemize}
\item $\widehat{z}_{1}=\frac{z_{1}-z_{1}(q_{i})}{h_{\omega_{i}, q_{i}}}$,..., $\widehat{z}_{n}=\frac{z_{n}-z_{n}(q_{i})}{h_{\omega_{i}, q_{i}}}$, denote the defined inverse map from  $B_{\widehat{0}}(\frac{1}{\epsilon_{i}})\subset\C^{n}$ to $B(1)$ as $\Gamma_{i}$;
\item  $\widehat{\omega}_{i}=\frac{1}{h^{2}_{\omega_{i},q_{i}}}\Gamma_{i}^{\star}\omega_{i}$,  $\widehat{F}_{i}=\Gamma_{i}^{\star}F_{i}$. 
\end{itemize}
Denote the Euclidean metric   with respect to the new coordinates $(\widehat{z}_{1},..., \widehat{z}_{n})$ as $\widehat{\omega}_{Euc}$. Thus,   in $B_{\widehat{0}}(\frac{1}{\epsilon_{i}})$ with respect to the new coordinates, then following holds.
\begin{equation}\label{equ MA equation holds in rescaled coordinates}
det \widehat{\omega}_{i,\widehat{k}\bar{\widehat{l}}}=e^{\widehat{F}_{i}}.
\end{equation}

Moreover, the H\"older radius of $\widehat{\omega}_{i}$ is 1 at the origin i.e
\begin{equation}\label{equ holder radius at the origin is 1 after rescaling}
h_{\widehat{\omega}_{i},\widehat{0}}=1.
\end{equation}
From now on, we add \  $\widehat{\cdot}$\  to those objects in the rescaled coordinates, so the reader can figure out that everything with\   $\widehat{\cdot}$\  is after rescaling. 

 For any $\infty>\lambda>0$, when $i$ is large enough, the metrics $\widehat{\omega}_{i}$  live  in $B_{\widehat{0}}(\lambda)$ in the rescaled coordinates.  For any $\widehat{p}$ in  $\C^{n}$, when $i$ is large enough such that $\widehat{p}\in B_{\widehat{0}}(\frac{1}{2\epsilon_{i}})$, consider the preimage of $\widehat{p}$ under the rescaling map as  \[p_{i}=\Gamma_{i}(\widehat{p})=\widehat{p}h_{\omega_{i}, q_{i}}+q_{i}\  \textrm{with respect to the coordinates}\  (z_{1},...,z_{n}).\] 

By the 4th item in (\ref{eqn contradiction hypothesis}),  we have $h_{\omega_{i}, p_{i}}\geq \frac{h_{\omega_{i}, q_{i}}d_{p_{i}}}{2 d_{q_{i}}}$.
Then after rescaling (with the factor $\frac{1}{h_{\omega_{i}, q_{i}}}$), we have 
\begin{equation}\label{equ holder radius big after rescaling}
h_{\widehat{\omega}_{i},\widehat{p}}\geq \frac{d_{p_{i}}}{2 d_{q_{i}}}.
\end{equation}

Notice  $\frac{d_{p_{i}}}{ d_{q_{i}}}$ is invariant under rescaling i.e
\begin{equation}\label{equ rescaling invariance of dist quotients}
\frac{d_{p_{i}}}{ d_{q_{i}}}=\frac{dist_{\widehat{\omega}_{Euc}}(\widehat{p},\partial \widetilde{B})}{dist_{\widehat{\omega}_{Euc}}(\widehat{0},\partial \widetilde{B})}=\frac{dist_{\widehat{\omega}_{Euc}}(\widehat{p},\partial \widetilde{B})}{\frac{1}{\epsilon_{i}}},
\end{equation}
where $\widetilde{B}$ is the image of $B(1)$ under the rescaling map.  Since  $$dist_{\widehat{\omega}_{Euc}}(\widehat{p},\widehat{0})<\infty,$$ then (\ref{eqn contradiction hypothesis}) and (\ref{equ rescaling invariance of dist quotients}) imply
\begin{equation}\label{equ dist quotient tends to 1}\lim_{i\rightarrow \infty}\frac{d_{p_{i}}}{ d_{q_{i}}}=1.\end{equation}
Therefore when $i$ is large, (\ref{equ holder radius big after rescaling}) and (\ref{equ dist quotient tends to 1}) imply $h_{\widehat{\omega}_{i},\widehat{p}}>\frac{1}{3}$. 

Hence, by Claim \ref{clm good property of the radius complex case}, we have 
\begin{equation}\label{equ good potential after rescaling}
  [\widehat{\omega}_{i}]_{\alpha,B_{\widehat{p}}(\frac{1}{3})}\leq 3^{\alpha}\delta_{0}.
\end{equation}
Choosing  $\omega_{c}=\widehat{\omega}_{i}(\widehat{p})$, and $\delta_{0}$ small enough with respect to $K$, (\ref{equ good potential after rescaling}) implies the small oscillation condition in Proposition \ref{prop CDS trick} is fulfilled in $B_{\widehat{p}}(\frac{1}{3})$. Then applying Proposition \ref{prop CDS trick} (rescaled to $B_{\widehat{p}}(\frac{1}{3})$), we end up with 
\begin{equation}\label{equ bootstrapping for the potential}
[\widehat{\omega}_{i}]_{\alpha^{\prime},B_{\widehat{p}}(\frac{1}{20})}\leq C.
\end{equation}
Then, (\ref{equ bootstrapping for the potential}) and the second item in (\ref{eqn contradiction hypothesis}) imply that for any $\lambda>0$, when $i$ is large enough such that $\frac{1}{\epsilon_{i}}>
1000(R +1)$,  the following crucial bootstrapping estimate holds:
\begin{equation}\label{equ crucial bootstrapping estimate}
|\widehat{\omega}_{i}|_{\alpha^{\prime},B_{\widehat{0}}(R)}\leq C.
\end{equation}

Step 3: Strong convergence of the rescaled sequence, rigidity of bubble, and contradiction. 

Then, by the Arela-Ascoli theorem, the sequence $\widehat{\omega}_{i}$ subconverges to an  $(\widehat{\omega}_{\infty}, \mathbb{C}^n)$ in $C^{\widehat{\alpha}}(B(\lambda))$-topology, for any $\lambda>0,\ \alpha<\widehat{\alpha}<\alpha^{\prime}$. In particular, we have 
 \begin{equation}\label{equ strong convergence of omega over compact sets}\lim_{i\rightarrow \infty}|\widehat{\omega}_{i}-\widehat{\omega}_{\infty}|_{\alpha, B_{\widehat{0}}(200)}=0.\end{equation}

 By the hypothesis that   $|F_{i}|_{\alpha^{\prime}, B(1)}\leq c$ in (\ref{eqn contradiction hypothesis}), and the hypothesis that $\frac{1}{h_{\omega_{i}, q_{i}}}\rightarrow \infty$, the pulled back functions $\widehat{F}_{i}$ subconverges to a constant $C_{1}$ in $C^{\widehat{\alpha}}[B(\lambda)]$-topology for any $\lambda>0$.  Then  the following holds on $\widehat{\omega}_{\infty}$. 
\begin{itemize}
\item As a form,  $\widehat{\omega}_{\infty}\in C^{\widehat{\alpha}}(\mathbb{C}^{n})$, for all $0<\widehat{\alpha}<\alpha^{\prime}$;

\item $det\widehat{\omega}_{\infty,\widehat{k}\bar{\widehat{l}}}=e^{C_1}$ over $\mathbb{C}^n$;

\item $\frac{\widehat{\omega}_{Euc}}{K}\leq \widehat{\omega}_{\infty}\leq K\widehat{\omega}_{Euc}$.
\item  $\widehat{\omega}_{\infty}$ admits potential over any finite ball, therefore $\widehat{\omega}_{\infty}$ is closed.  To see this, for any ball $\lambda>0$, applying (\ref{equ crucial bootstrapping estimate}) and  Lemma \ref{lem dbar solvability smooth case} to $B_{\widehat{0}}(100\lambda+100)$, we obtain  potentials $\widehat{\phi}_{i, \lambda}$ such that 
\begin{equation}
\widehat{\omega}_{i}=\sqrt{-1}\partial \bar{\partial}\widehat{\phi}_{i, \lambda},\   |\widehat{\phi}_{i,\lambda}|_{2,\alpha^{\prime},B_{\widehat{0}}(\lambda+1)}\leq C_{\lambda}\ \textrm{over}\ B(\lambda+1).
\end{equation}
Then, $\widehat{\phi}_{i,\lambda}$ subconverges (strongly) in $C^{2,\alpha}[B_{\widehat{0}}(\lambda)]-$topology to a potential $\widehat{\phi}_{\infty, \lambda}$ such that 
\begin{equation}\label{equ limit metric admits potential over finite balls}
\widehat{\omega}_{\infty}=\sqrt{-1}\partial \bar{\partial}\widehat{\phi}_{\infty, \lambda} \ \textrm{over}\ B_{\widehat{0}}(\lambda),\ 
|\widehat{\phi}_{\infty,\lambda}|_{2,\alpha,B_{\widehat{0}}(\lambda)}\leq C_{\lambda}.
\end{equation}

\end{itemize}

  Thus, the above $4$ items imply the conditions in Theorem \ref{thm Liouville} are fulfilled. According to Theorem \ref{thm Liouville}, $\widehat{\omega}_{\infty}$ is of constant coefficients, thus 
\begin{equation}\label{equ limit has no oscillation}[\widehat{\omega}_{\infty}]_{\alpha, B_{\widehat{0}}(200)}=0.
\end{equation}
Hence (\ref{equ strong convergence of omega over compact sets}) and (\ref{equ limit has no oscillation}) imply 
 \begin{equation}\label{equ limit of ossilation is  0}\lim_{i\rightarrow \infty}[\widehat{\omega}_{i}]_{\alpha, B_{\widehat{0}}(200)}=0.\end{equation}
Then when $i$  is large enough, we deduce
$$[\widehat{\omega}_{i}]_{\alpha, B_{\widehat{0}}(100)}\leq \frac{\delta_{0}}{100^{\alpha}}.$$
 This means \[h_{\widehat{\omega}_{i},\widehat{0}}\geq 100,\]
 which contradicts (\ref{equ holder radius at the origin is 1 after rescaling}) !
 
 The proof of Theorem \ref{thm of Yu Wang} is completed. 
 \end{proof}
\begin{rmk}Actually, in the item containing (\ref{equ limit metric admits potential over finite balls}) in Step 3 of the above proof,  to prove the existence of  potentials for $\widehat{\omega}_{\infty}$ over all finite balls , it is easier to prove first by definition that $\omega$ is a closed current, and then apply Lemma \ref{lem dbar solvability smooth case}. However, since we want to carry our proof  in this section exactly and directly to section \ref{A new proof of the aprori version of  Cafferelli's  estimate} without involving more issues,   we still want to take  $\widehat{\phi}_{\infty, \lambda}$ as the  limit of the potentials of $\widehat{\omega}_{i}$.  
\end{rmk}

\section{Appendix:  A new proof of the aprori version of  Cafferelli's  estimate for real Monge-Ampere equations. \label{A new proof of the aprori version of  Cafferelli's  estimate}}
The proof of Theorem \ref{thm of Caffarelli} is exactly parallel to the proof in section \ref{A new proof of the aprori version of Yu-Wang's Cafferelli type estimate}. Namely, to translate the "complex" proof in section \ref{A new proof of the aprori version of Yu-Wang's Cafferelli type estimate} to the real case of Theorem \ref{thm of Caffarelli}, we only need to 
\begin{itemize}
\item  change the $"\sqrt{-1}\partial\bar{\partial}"$ in section \ref{A new proof of the aprori version of Yu-Wang's Cafferelli type estimate} to $"\nabla^{2}"$ (Hessian);
\item change the $\omega_{k\bar{l}}$ to $g_{kl}$, $\phi_{k\bar{l}}$ to $u_{kl}$;
\item change the complex coordinates $"z_{1}...z_{n}"$ in  section \ref{A new proof of the aprori version of Yu-Wang's Cafferelli type estimate} to real coordinates $"x_{1}...x_{n}"$;
\item change the words "plurisubharmonic" to "convex";
\item change the equation from (\ref{equ MA}) to (\ref{equ real MA}). 

\end{itemize} 

By translating as above, Lemma \ref{lem dbar solvability smooth case} corresponds to Lemma \ref{lem hessian solvability smooth case}, Theorem \ref{thm Liouville} corresponds to Theorem \ref{thm Calabi Liouville}, Proposition \ref{prop CDS trick} corresponds to Proposition \ref{prop CDS trick real case}. One thing worth mentioning is, while the proof of Lemma \ref{lem dbar solvability smooth case} requires Griffith-Harris' trick \cite{GH} and Hormander's results \cite{Hormander}, Lemma \ref{lem hessian solvability smooth case} can be proved in one line. 

 \begin{lem}\label{lem hessian solvability smooth case}Suppose $\frac
{1}{\Lambda}<r<\Lambda$ for some $\Lambda>0$, then there exists a constant $C_{\Lambda}$ depending on $\Lambda$ and $n$ with the following 
properties. 

  Suppose $g$ is a matrix-valued function over $B_{r}$, such that $g=\nabla^{2}u$ for some  function $u\in C^{2,\alpha}(B_{r})$. Then there exists a function $v\in C^{2,\alpha}[B(\frac{r}{2})]$ such that $g=\nabla^{2}v$ and 
 $$|v|_{0, B_{\frac{r}{2}}}\leq Cr^2|g|_{0,B_r},$$
where $C$ is constant depending on $n$. Consequently, 
\[|v|_{2,\alpha,B_{\frac{r}{4}}}\leq C_{\Lambda}|g|_{\alpha,B_{r}}.\]
 \end{lem}
\begin{proof}{of Lemma \ref{lem hessian solvability smooth case}:} The proof can not be easier. Just take $v$ as $u$ minus its linearization i.e 
\begin{equation}
v=u-u(0)-x\cdot \nabla u(0),
\end{equation}
then 
\[\nabla^{2}v=g,\ v(0)=0,\ (\nabla v)(0)=0.\]
Thus the estimate of $|v|_{L^{\infty}[B(\frac{r}{2})]}$ follows by applying the mean value theorem to $\nabla v$ and then to $v$.
\end{proof}

\begin{thm}\label{thm Calabi Liouville}(Calabi \cite{Calabi58}) (Pogorelov \cite{Pogorelov}) Suppose $g$ is a symmetric-matrix-valued function  defined over $\R^{n}$. Suppose $g$ admits a $C^{2,\alpha}-$potential over any finite ball i.e for any ball $B\in \R^{n}$, there exists a function $u_{B}\in C^{2,\alpha}(B)$ such that 
$$g=\nabla^{2}u_{B}\ \textrm{over}\ B.$$ 
Suppose there is a constant  $K$ such that  
\begin{equation}
det g_{kl}=1,\ \frac{1}{K}I\leq g_{kl}\leq KI\ \textrm{over}\ \R^{n}.
\end{equation}
Then, for any $1\leq k,l\leq n$, $g_{kl}$ is a constant. 
\end{thm}

\begin{rmk} Actually Calabi's and Pogorelov's original theorems in \cite{Calabi58} and \cite{Pogorelov} are much stronger than Theorem \ref{thm Calabi Liouville}, but all we need here is Theorem \ref{thm Calabi Liouville}. In \cite{Calabi58},  $g$ is assumed to admit a global potential. Though in our new proof of Theorem \ref{thm of Caffarelli} of Caffarelli, we have a global potential, we still want  to state the Liouville theorem  as Theorem \ref{thm Liouville} to emphasize that it  does not need a global potential. 
\end{rmk}

\begin{prop}\label{prop CDS trick real case} For any constant coefficient Riemannian metric $g_{c}$, there exist a small enough positive number $\delta$ and a big enough constant $C_{g_{c}}$, both depending on the positive lower and upper bounds on the eigenvalues of $g_{c}$, the dimension $n$, and $\alpha^{\prime}$, with the following properties. Suppose $u$ is a $C^{2,\alpha^{\prime}}$ convex function  defined over
$B_{0}(1)$  such that  
\begin{equation}
det u_{ij}=e^{f},\  \frac{g_{c}}{1+\delta}\leq \nabla^{2}u\leq (1+\delta)g_{c}\ \textrm{over}\ B_{0}(1),
\end{equation}
then the following estimate holds in $B(\frac{1}{4})$. 
\[[\nabla^{2}u]_{\alpha^{\prime}, B(\frac{1}{4})}\leq 
C_{g_c}(|e^f|_{\alpha^{\prime},B(1)}+|\nabla^{2} u|_{0,B(1)}).\]
\end{prop}

With the above discussion in section \ref{A new proof of the aprori version of  Cafferelli's  estimate}, the proof of Theorem \ref{thm of Caffarelli} is complete.

Xiuxiong Chen, Department of Mathematics, Stony Brook University,
NY, USA;\ \ xiu@math.sunysb.edu.\\

Yuanqi Wang, Department of Mathematics, University of California  at Santa Barbara, Santa Barbara,
CA,  USA;\ \ wangyuanqi@math.ucsb.edu.

  \end{document}